\documentclass[english]{article}
\usepackage{amsmath,amssymb,amstext,babel,graphicx}
\makeatletter \@addtoreset{equation}{section}

\makeatother \makeatletter

\newcommand{\qev}{\hfill{\mbox{$\square$}}\medskip}

\newcommand{\N}{\ensuremath{\mathbb{N}}}
\newcommand{\Z}{\ensuremath{\mathbb{Z}}}
\newcommand{\R}{\ensuremath{\mathbb{R}}}
\newcommand{\C}{\ensuremath{\mathbb{C}}}

\renewcommand{\Re}{{\rm Re}\,}

\newcommand{\proof}{\noindent{\sc Proof.\ }}
\newcommand{\e}{e}

\newtheorem{definition}{Definition}[section]

\newtheorem{theorem}[definition]{Theorem}
\newtheorem{proposition}[definition]{Proposition}

\newtheorem{remark}[definition]{Remark}

\newtheorem{example}[definition]{Example}

\title{On the interpolation of discontinuous functions}
\author{Michele Campiti, Giusy Mazzone, Cristian Tacelli\thanks{Department of Mathematics ``E. De Giorgi, University of
Salento, P.O.Box 193, 73100 Lecce, Italy. E-mails: michele.campiti@unisalento.it, giusy.mazzone@unisalento.it, cristian.tacelli@unisalento.it.''}}
\date{}

\makeatother

\begin{document}

\maketitle

\pagestyle{plain}

\bigskip{}

\noindent \textit{AMS Classification (2000): 41A05, 11M35, 40A30, 54A20}

\bigskip{}

\noindent \textit{Keywords: Interpolation, Discontinuous Functions, Statistical Convergence, Lerch Zeta Function, Hurwith Zeta Function}

\begin{abstract}
Given a sequence of real numbers, we consider its subsequences converging to possibly different
limits and associate to each of them an index of convergence which depends on the density
of the associated subsequences. This index turns out to be useful for a complete description
of some phenomena in interpolation theory at points of discontinuity of the first kind.
In particular we give some applications to Lagrange and Shepard operators.
\end{abstract}

\section{An index of convergence}

The aim of this paper is to investigate the behavior of non converging sequences,
for which we can find suitable converging subsequences. The density of the subsequences
converging to a given limit determines an index of convergence in the sense of
Definition \ref{df:index}. Our main aim is to use this index in order to obtain
a complete description of the behavior of some sequences of interpolating operators
on functions with a finite number of discontinuity of the first kind. This problem
has been considered for a long time both for algebraic and trigonometric polynomials.
While for trigonometric polynomials we have some classical completely satisfactory results,
in the case of algebraic polynomials the situation is quite different.
Some properties of Shepard operators on functions with a discontinuity of the first kind
have been established in \cite{BDVM} and subsequent papers in terms of lower and upper limits, but the problem of a complete description remains substantially opened.

The introduction of the index of convergence in Definition \ref{df:index} allows us to give a solution to this problem. One of the main properties of this index resides in the fact that a sequence may converge to different real numbers having indices in the interval $[0,1]$.
In the case where only one real number has index 1, the concept can be related to that of statistical convergence considered in \cite{F1} and subsequently generalized in different ways (see, e.g., \cite{F2, F3, C1, DO}). Due to the particular formulation of the concept of statistical convergence and its subsequent extensions and generalizations, it has not been possibile to use it for a deeper analysis of the interpolation of discontinuous functions.

In this section we define the index of convergence and give some of its properties
and characterizations. In Sections \ref{sc:Lagrange} and \ref{sc:Shepard} we consider
the indices of convergence of Lagrange and respectively Shepard operators applied to functions having
a finite number of points of discontinuity of the first kind.

Let $K\subset \N$; the lower density and, respectively, the upper density of $K$ are defined by
\[
\delta_-(K):=\liminf_{n\to+\infty}\frac{\left|K\cap\{1,\dots,n\}\right|}{n}\,,\quad
\delta_+(K):=\limsup_{n\to+\infty}\frac{\left|K\cap\{1,\dots,n\}\right|}{n}\,.
\]
In the case where $\delta_-(K)=\delta_+(K)$ the density of $K$ is defined as follows
\[
\delta(K):=\delta_-(K)=\delta_+(K)\;.
\]

We observe that $\delta_-(K)=1-\delta_+(K^c)$. Indeed
\begin{align*}
&\delta_-(K)=\liminf_{n\to+\infty}\frac{\left|K\cap\{1,\dots,n\}\right|}{n}
=\liminf_{n\to+\infty}\frac{\left|K\cap\{1,\dots,n\}\right|+n-n}{n}\\
&\quad =\liminf_{n\to+\infty}\left(1-\frac{n-\left|K\cap\{1,\dots,n\}\right|}{n}\right)
=1+\liminf_{n\to+\infty}\left(-\frac{\left|K^c\cap\{1,\dots,n\}\right|}{n}\right)\\
&\quad =1-\limsup_{n\to+\infty}\left(\frac{\left|K^c\cap\{1,\dots,n\}\right|}{n}\right)=1-\delta_+(K^c).
\end{align*}

Similarly, it can be shown that $\delta_+(K)=1-\delta_-(K^c)$.

%For a sequence $(x_n)_{n\in\N}$ of real numbers we define the \emph{outer convergence measure}
%$m^*$ by setting, for every subset $A$ of $\R$,
%\[
%m^*(A):=\delta_+\left( \left\{n\in\N\ |\ x_n\in A \right\}\right)\;.
%\]
%Hence, with a straightforward argument, we can define the $\sigma$-algebra
%\[
%{\cal M}:=\{A\subset \R\ |\ m^*(A)+m^*(A^c)\le 1\}\;,
%\]
%and the probability measure $m:{\cal M}\rightarrow\R$ by
%\begin{equation} \label{eq:measure}
%m(A):=m^*(A)\;,\qquad A\in{\cal M}\;.
%\end{equation}
%
%The measure $m:{\cal M}\rightarrow\R$ is called the \emph{probability convergence measure associated with the %sequence }$(x_n)_{n\in\N}$.
%
%We are now in a position to make the following definition.
%
%\begin{definition} \label{df:index}
%Let $(x_n)_{n\in\N}$ be a sequence of real numbers.
%For every real number $L$, the \emph{index of convergence of the sequence} $(x_n)_{n\in\N}$
%\emph{to} $L$ is defined by
%\[
%i\left( x_n;L \right):=1-\sup_{\varepsilon>0}m^*\left(]-\infty,L-\varepsilon]\cup
%[L+\varepsilon,+\infty[\right)\;.
%\]
%Moreover, we also set
%\[
%i\left( x_n;+\infty \right):=1-\sup_{M\in\R}m^*\left(]-\infty,M]\right)\;,
%\
%i\left( x_n;-\infty \right):=1-\sup_{M\in\R}m^*\left([M,+\infty[\right)\;.
%\]
%If $L\in\R$ or $L=\pm\infty$, we shall use the notation $\textit{st}(\alpha)-\lim_{n\to+\infty}x_n=L$
%to indicate that $i\left( x_n;L \right)=\alpha$.
%\end{definition}
%
We are now in a position to make the following definition.

\begin{definition} \label{df:index}
Let $(x_n)_{n\geq1}$ be a sequence of real numbers.
For every real number $L$, the \emph{index of convergence of the sequence} $(x_n)_{n\geq1}$
\emph{to} $L$ is defined by
\[
i\left( x_n;L \right):=1-\sup_{\varepsilon>0}\delta_+\left(\left\{n\in \N\ |\ x_n\in ]-\infty,L-\varepsilon]\cup
[L+\varepsilon,+\infty[\right\}\right)\;.
\]
Moreover, we also set
\begin{align*}
i\left( x_n;+\infty \right)&:=1-\sup_{M\in\R}\delta_+\left(\left\{n\in \N\ |\ x_n\in ]-\infty,M]\right\}\right)\;,
\\
i\left( x_n;-\infty \right)&:=1-\sup_{M\in\R}\delta_+\left(\left\{n\in \N\ |\ x_n\in[M,+\infty[\right\}\right)\;.
\end{align*}
%If $L\in\R$ or $L=\pm\infty$, we shall also use the notation $\textit{st}(\alpha)-\lim_{n\to+\infty}x_n=L$ to %indicate that $i\left( x_n;L \right)=\alpha$.
\end{definition}

\begin{remark} \label{rm:chardelta-} \rm
We point out the following explicit expression of the index of convergence of a sequence $(x_{n})_{n\geq1}$
\begin{align*}
i(x_{n};L)&=1-\sup_{\varepsilon>0}\delta_{+}\left(\{n\in \N\ |\ x_{n}\in ]-\infty,L-\varepsilon]
\cup[L+\varepsilon, +\infty[\} \right)=
\\
&=1+\inf_{\varepsilon>0}\left(-\delta_{+}\left(\{n\in \N\ |\ x_{n}\in ]-\infty,L-\varepsilon]
\cup[L+\varepsilon,+\infty[\}
\right)\right)=
\\
&=\inf_{\varepsilon>0}\left(1-\delta_{+}\left(\{n\in \N\ |\ x_{n}\in ]-\infty,L-\varepsilon]
\cup[ L+\varepsilon,+\infty[\}
\right)\right)=
\\
&=\inf_{\varepsilon>0}\delta_{-}\left(\{n\in \N\ |\ x_{n}\in ]L-\varepsilon,L+\varepsilon[\}\right).
\end{align*}
\end{remark}

Definition \ref{df:index} can be extended as follows.

We set for brevity $B_\varepsilon:=]-\varepsilon,\varepsilon[$ whenever $\varepsilon>0$.

\begin{definition}
Let $(x_{n})_{n\geq1}$ be a sequence of real numbers and let $A$ be a subset of $\R$. We define \emph{index
of convergence of} $(x_{n})_{n\geq1}$ \emph{relatively to} $A$
\[
i(x_{n},A):=1-\sup_{\varepsilon>0}\delta_{+}(\{n\in \N\ |\ x_{n}\notin A+B_\varepsilon\}).
\]
\end{definition}

Also in this case we have the following expression of the index of convergence
\[
i(x_{n},A)=\inf_{\varepsilon>0}\delta_{-}(\{n\in \N\ |\ x_{n}\in A+B_\varepsilon\})\;.
\]

\begin{example} \label{ex:cos_equi} \rm
\begin{itemize}
\item[i)]
As a simple example, we can take $x_n:=\cos n\pi/2$, it is easy to recognize that
%\[
%\textit{st}\left( \frac{1}{2} \right)-\lim_{n\to+\infty} x_n=0,\
%\textit{st}\left( \frac{1}{4} \right)-\lim_{n\to+\infty} x_n=1,\
%\textit{st}\left( \frac{1}{4} \right)-\lim_{n\to+\infty} x_n=-1.
%\]
\[
i(x_n;0)=\frac 12,\ \ i(x_n;1)=\frac 14,\ \ i(x_n;-1)=\frac 14.
\]

\item[ii)]
As a further example, let $\alpha \in [0,1[$ be irrational, $\beta \in [0,1[$ and consider
\[
x_n:=n\alpha +\beta-[n\alpha+\beta]\quad(=n\alpha +\beta\mod 1)
\]
where $[x]$ denotes the integer part of $x$.

The well-known equidistribution theorem of Weyl ensures that
\[
 \lim _{n\to \infty}\frac 1n \sum _{k=0}^{n-1}f(k\alpha+\beta\mod 1)=
\int_0^1f(t)dt
\]
for each Riemann integrable function in $[0,1]$. Then
it follows that
\[\delta \left(\left\{n\in \N\ |\ x_n\in A\right\}\right)=|A|\]
for every Peano-Jordan measurable set $A\subset [0,1[$, where $|\cdot|$ denotes the Peano-Jordan measure.
Then
\[
i(x_n;A)=|A|.
\]
\end{itemize}
\end{example}

\begin{remark} \label{rm:particular_cases} \rm
If the index of convergence of a sequence $(x_{n})_{n\geq1}$ to a real number $L$ is equal to 1,
we have
$$
\sup_{\varepsilon>0}\delta_{+}\left(\{n\in \N\ |\  x_{n}\in ]-\infty,L-\varepsilon]\cup[L+
\varepsilon,+\infty[\}\right)=0.
$$
Hence, for all $\varepsilon>0$
\begin{align*}
0&=\delta_{+}\left(\{n\in \N\ |\ x_{n}\in ]-\infty,L-\varepsilon]\cup[L+
\varepsilon,+\infty[\}\right)
\\
&\geq \delta_{-}\left(\{n\in \N\ |\ x_{n}\in ]-\infty,L-\varepsilon]\cup[L+
\varepsilon,+\infty[\}\right)\geq 0\;,
\end{align*}
consequently
$\delta\left(\{n\in \N\ |\ x_{n}\in ]-\infty,L-\varepsilon]\cup[L+
\varepsilon,+\infty[\}\right)=0$,
%If we set $K_{\varepsilon}:=\{n\in \N\ |\ x_{n}\in ]-\infty,L-\varepsilon]\cup[L+
%\varepsilon,+\infty[\}$, we get
%\[
%\delta_{+}(K_{\varepsilon})=\delta_{-}(K_{\varepsilon})=\delta(K_{\varepsilon})=0
%\]
which means that $(x_{n})_{n\geq1}$ converges statistically to $L$.

%In the other extreme case where $i(x_{n};L)=0$, we obtain
%\[
%\sup_{\varepsilon>0}\delta_+\left(\left\{n\in \N\ |\ x_n\in %]-\infty,L-\varepsilon]\cup[L+\varepsilon,+\infty[
%\right\}\right)=1;
%\]
%so there exists $\tilde{\varepsilon}>0$ such that $\delta_+\left(\left\{
%n\in \N\ |\ x_n\in ]-\infty,L-\tilde{\varepsilon}]\cup[L+\tilde{\varepsilon},+\infty[\right\}\right)=1$.
\end{remark}

In the next proposition we point out some relations between the index of convergence to a number $L$ and the density of suitable subsequences converging to $L$.

\begin{proposition}\label{prop:char_index}
Let $(x_n)_{n \geq1}$ be a sequence of real numbers and $\sigma\in]0,1]$. Then
$i(x_n,L)\geq \sigma$ if and only if there exists a subsequence $(x_{k(n)})_{n\geq1}$ converging to $L$ such that
\[
\delta_{-} \left(\{k(n)\;|\;n\in \N\} \right)\ge \sigma\;.
\]
\end{proposition}

\proof $\Rightarrow$)\quad
For every $n\ge1$, we consider the set $M_{1/n}:=\{m\in\N\ |\ |x_m-L|<1/n\}$.
{}From Remark \ref{rm:chardelta-}, for every $n\in \N$ there exists $\tilde \nu_n$ such that
\[
\frac{\left| M_{1/n}\cap \{1,2,\dots,j\}\right|}{j}\geq \sigma-\frac{1}{n}
\]
whenever $j>\tilde \nu _n$. At this point we define recursively a new sequence $(\nu_n)_{n\ge1}$
by setting $\nu_1=\tilde \nu_1$ and $\nu_n=\max \{\tilde \nu_n,\nu_{n-1}+1\}$. We have
\begin{equation}
  \label{eq:density-n2}
\frac{\left| M_{1/n}\cap \{1,2,\dots,j\}\right|}{j}\geq \sigma-\frac{1}{n}
\text{ for all } j>\nu_n\;.
\end{equation}

Consider the set of integers
\[
K=\bigcup_{n\ge1} \left( M_{1/n}\cap \{1,2,\dots,\nu_{n+1}\}\right)
\]
and the subsequence $(x_n)_{n\in K}$.

For every $\varepsilon >0$, let $m\in \N$ such that $1/m\leq \varepsilon$.
Then for every $k\in K$ satisfying $k>\nu_ m$
we have
$k\in  \bigcup_{n\geq  m} \left( M_{1/n}\cap \{1,2,\dots,\nu_{n+1}\}\right)$
and hence $|x_k-L|< \frac{1}{m}\leq \varepsilon$. This shows that the subsequence
$(x_n)_{n\in K}$ converges to $L$.

On the other hand, for every $j>\nu_m$, there
exists $l\geq m$ such that $\nu_l<j\leq \nu _{l+1}$ and
 thanks to \eqref{eq:density-n2}
we have
\begin{align*}
\frac{|K\cap\{1,2,\dots,j\}| }{j}&\geq \frac{\left| M_{1/l}\cap \{1,2,\dots,\nu_{l+1} \}\cap \{1,2,\dots,j\}\right|}{j}\\
&=\frac{\left| M_{1/l}\cap \{1,2,\dots, j\}\right|}{j}\geq \sigma-\frac{1}{l}\geq \sigma-\frac{1}{m} \geq \sigma-\varepsilon
\end{align*}
that is
\[
\liminf_{n \to \infty}\frac{|K\cap\{1,2,\dots,j\}| }{j}\geq \sigma\;.
\]

\noindent $\Leftarrow$)\quad
We suppose that there exists
a subsequence $(x_{k(n)})_{n\geq1}$ converging to $L$
such that such that
$\delta_{-} \left(\{k(n)\;|\;n\in \N\} \right)\ge \sigma$.
For every $\varepsilon > 0$ there exists $\nu_\varepsilon\in \N$ such
that $|x_{k(n)}-L|<\varepsilon$ whenever $n\geq \nu_\varepsilon$. Hence
\begin{align*}
\delta_-(\{n\in \N\;|\; |x_n-L|<\varepsilon\})&\geq
\delta_-(\{n\in \N\;|\; |x_{k(n)}-L|<\varepsilon\})\\
&=\delta_-(\{k(n)\;|\; n\geq \nu_\varepsilon\})\\
&=\delta_-(\{k(n)\;|\; n\in \N\})\ge
\sigma
\end{align*}
and therefore, from Remark \ref{rm:chardelta-}, we obtain $i(x_n,L)\geq \sigma$.
\qev

\begin{proposition}\label{pr:sum_gen}
Let $(x_{n})_{n\geq1}$ be a sequence of real numbers and $(A_m)_{m\ge1}$ a sequence of subsets of $\R$ such that
$\overline{A_{k}}\cap \overline{A_{j}}=\emptyset$ for all $k\neq j$. Then
\[
0\leq \sum^{+\infty}_{k=1}i(x_{n},A_{k})\leq 1.
\]

In particular, if $(L_m)_{m\geq1}$ is a sequence of distinct elements of $[-\infty,\infty]$ such that, for every
$m\geq 1$
\begin{align*}
i(x_n;L_m)=\alpha_m\;,
\end{align*}
for some $\alpha_{m}\ge0$, then
\[
0\leq \sum^{+\infty}_{k=1}\alpha_{k}\leq 1.
\]
\end{proposition}

\begin{proof}
Let $N\ge1$; since $\overline{A_{k}}\cap \overline{A_{j}}=\emptyset$ whenever $k\neq j$, we can choose
$\varepsilon$ such that
\[
(A_{k}+B_\varepsilon)\cap (A_{j}+B_\varepsilon)=\emptyset
\]
for all $k,j=1,\dots,N$, $k\neq j$.

Now consider the set
\[
M^{(k)}_{\varepsilon}:=\{n\in \N\ |\ x_{n}\in A_{k}+B_\varepsilon\}
\]
and observe that $M^{(k)}_{\varepsilon}\cap M^{(j)}_{\varepsilon}=\emptyset$
whenever $k,j=1,\dots,N$, $k\neq j$. Then we can conclude that
\begin{align*}
0\leq &\sum^{N}_{k=1}i(x_{n},A_{k})\leq \sum^{N}_{k=1}\delta_{-}(\{n\in \N\ |\ x_{n}\in A_{k}+B_\varepsilon
\})
\\
&=\sum^{N}_{k=1}\liminf_{n\rightarrow \infty}\frac{|M^{(k)}_{\varepsilon}\cap\{1,\dots,n\}|}{n}
\leq
\liminf_{n\rightarrow \infty}\left( \sum^{N}_{k=1}\frac{|M^{(k)}_{\varepsilon}\cap\{1,\dots,n\}|}{n} \right)
\\
&=\liminf_{n\rightarrow \infty}\frac{\left|\bigcup^{N}_{k=1}M^{(k)}_{\varepsilon}\cap\{1,\dots,n\}\right|}{n}
=\delta_{-}\left(\bigcup^{N}_{k=1}M^{(k)}_{\varepsilon}\right)\leq 1
\end{align*}
\qev
\end{proof}

\begin{remark} \label{rm:sum=1} \rm
Observe that if in the preceding Proposition we have $\sum_{k=1}^{+\infty}\alpha_{k}=1$, then every
subsequence $\left(x_{k(n)}\right)_{n\geq1}$ of $(x_{n})_{n\geq1}$
which converges to a limit $L$ different from $L_m,$\ $m\ge1$,
necessarily satisfies $\delta_{-}(\{k(n)\ |\ n\in \N\})=0$ and therefore $i(x_n;L)=0$.

Indeed, if
a subsequence $\left(x_{k(n)}\right)_{n\geq1}$ of $(x_{n})_{n\geq1}$ exists such that
$\delta_{-}(\{k(n)\ |\ n\in \N\})=\alpha>0$,
then by Proposition \ref{prop:char_index}, we get $i(x_{n},L)\geq \alpha$ and therefore
\[
\sum^{N}_{k=1}i(x_{n},L_{k})+i(x_{n},L)\geq \sum^{N}_{k=1}\alpha_{k}+\alpha>1
\]
which contradicts Proposition \ref{pr:sum_gen}.
\end{remark}

\section{Lagrange operators on discontinuous functions} \label{sc:Lagrange}

We consider the classical Lagrange operators at the Chebyshev nodes.

The $n$-th Lagrange operator %$L_n:C([-1,1])\to C([-1,1])$
is defined by
\[
L_nf(x)=\sum_{k=1}^n\ell_{n,k}(x)f(x_{n,k}),
\]
for every %$f\in C([-1,1])$
$f:\, [-1,1]\rightarrow \R$ and $x\in[-1,1]$, where for
$k=1,\ldots,n$
\[
x_{n,k}=\cos\theta_{n,k},\qquad\theta_{n,k}=\frac{(2k-1)\pi}{2n}\;,
\]
are the Chebyshev nodes and
\[
\ell_{n,k}(x)=\prod_{i\not=k}\frac{x-x_{n,i}}{x_{n,k}-x_{n,i}}
\]
are the corresponding fundamental polynomials.

Identifying the variable $x\in[-1,1]$ with $\cos\theta$, with
$\theta\in[0,\pi]$, the polynomials $\ell_{n,k}$ may also be
expressed in terms of the variable $\theta$ as follows
\[
\ell_{n,k}(\cos\theta)=\frac{(-1)^{k-1}}{n}\,\frac{\cos
n\theta}{\cos\theta-\cos\theta_{n,k}}\,\sin\theta_{n,k}\;.
\]

Our aim is to study the behavior of the sequence of Lagrange operators for a particular class of functions having a finite number of points of discontinuity of the first kind.

We begin to consider the function
$h_{x_0,d}:[-1,1]\rightarrow\R$ defined by
\begin{equation} \label{eq:def_h}
h_{x_0,d}(x):=\left\{
\begin{array}{ll}
0\;,\quad& x<x_0\;,\\
d\;,\quad& x=x_0\;,\\
1\;,\quad& x>x_0\;,
\end{array} \right. \qquad x\in[-1,1]\;,
\end{equation}
where $d$ is a fixed real number.

Before stating our main result, we need to introduce some zeta functions.
Firstly, consider the Hurwitz zeta function
\begin{equation} \label{eq:def_H}
\zeta(s,a):=\sum^{+\infty}_{n=0}\frac{1}{(n+a)^s}
\end{equation}
for all $s,a\in \C$ such that $\Re[s]>1$ and $\Re[a]>0$. The previous series is absolutely convergent and
its sum can be extended to a meromorphic function defined for all $s\neq 1$.

Moreover we need to consider also the Lerch zeta function
\[
\Phi(x,s,a):=\sum^{+\infty}_{n=0}\frac{\e^{2n\pi i x}}{(n+a)^s}
\]
where $x\in \R$, $a\in ]0,1]$,
$\Re[s]>1$ if $x\in \Z$ and $\Re[s]>0$ otherwise. In the special case $x=\frac12$, we obtain the Lerch zeta function
\[
J(s,a):=\Phi\left(\frac 12 ,s,a\right)=\sum^{+\infty}_{n=0}\frac{(-1)^n}{(n+a)^s}
\]
which is related to the Hurwitz zeta function by the following relation
\[
J(s,a)=\frac{1}{2^{s-1}}\, \zeta\left(s,\frac a2\right)-\zeta(s,a)
\]
for all $s,a\in \C$ such that $0<a\leq 1$ and $\Re[s]>1$.

In order to state our main result, we define the function $g:]0,1[\mapsto \R$ by setting
\[
g(x):=
\frac{ \sin \left(\pi x\right) }{\pi}\,
J(1,x)\;,\qquad \text{if } x\in ]0,1[\;.
\]

\begin{theorem} \label{th:Lagrange}
Let $x_0=\cos\theta_0\in]-1,1[$ and consider the functions $h:=h_{x_0,d}$ defined by \eqref{eq:def_h}. Then, the sequence of functions $\left( L_nh\right)_{n\ge1}$ converges uniformly to $h$
on every compact subsets of $[-1,1]\setminus\{x_0\}$.

As regards the behaviour of the sequence $(L_nh(x_0))_{n\ge1}$ we have
\begin{itemize}
\item[i)] If $\frac{\theta_0}{\pi}=\frac pq $ with $p,q\in \N$, $q\neq 0$ and $GCD(p,q)=1$, then
  \[
    i\left(L_nh(x_0);
    g\left( \frac{2m+1}{2q}\right)\right)=\frac1q\;,
    \quad m=0,\dots, q-1
  \]
if $q$ is odd and
  \[
    i\left(L_nh(x_0);d\right)=\frac1q\;,\qquad
    i\left(L_nh(x_0);
    g\left( \frac{m}{q}\right)\right)=\frac1q\;,
    \quad m=1,\dots, q-1
  \]
if $q$ is even.

\item [ii)] if $\frac{\theta_0}{\pi}$ is irrational and if
$A\subset \R$ is a Peano-Jordan measurable set,
then
\[
i\left(L_nh(x_0);A\right)=|g^{-1}(A)|\;,
\]
where $|\cdot|$ denotes the Peano-Jordan measure.
\end{itemize}
\end{theorem}

\begin{proof}
Let $a=\cos \theta_1\in [-1,x_0[$ and
$x=\cos \theta\in [-1,a]$;
for sufficiently large $n\ge1$ there exists  $k_0$ such that
\[
0<\theta_{n,k_0}\leq \theta_0<\theta_{n,k_0+1}
<\theta_1\leq \theta\leq \pi
\]
and therefore
\[
0<\cos \theta_0-\cos \theta_1\leq \cos \theta_{n,k_0}-\cos \theta\;.
\]

We have $L_nh(\cos \theta)
=\sum_{k=1}^{k_0-1}
  \ell_{n,k}(\cos \theta)
  +d\ell_{n,k_0}(\cos \theta)$ if $\theta_{n,k_0}=\theta_0$, and
$L_nh(\cos \theta)=\sum_{k=1}^{k_0}
  \ell_{n,k}(\cos \theta)
$ if $\theta_{n,k_0}<\theta_0$;
hence
\begin{align*}
&L_nh(\cos \theta)\\&\quad=\sum_{k=1}^{k_0}
  \frac{(-1)^{k-1}}{n}\,\frac{\cos n\theta}{\cos\theta-\cos\theta_{n,k}}\,\sin\theta_{n,k}
+(d-1)\chi_{\left\{\theta_{n,k_0}\right\}}(\theta_0)\ell_{n,k_0}(\cos \theta)
  \\
&\quad=\sum_{k=1}^{k_0}
  \frac{(-1)^{k}}{n}\,
\frac{\cos n\theta}{\cos\theta_{n,k}-\cos\theta}\,\sin\theta_{n,k}
+(d-1)\chi_{\left\{\theta_{n,k_0}\right\}}(\theta_0)\ell_{n,k_0}(\cos \theta)\;.
\end{align*}

The function $t\to \frac{\sin t}{\cos t-\cos \theta}$
is positive and monotone increasing on the interval $[0, \theta[$;
since $0<\theta_{n,k}<\theta_{n,k+1}<\theta$ for every  $1\leq k\leq k_0$, we have
\begin{align*}
|L_n h(x)|&=|L_n(h)(\cos \theta)|\\&\leq
  \left|
  \frac{\cos n\theta}{n}\,\frac{\sin\theta_{n,k_0}}{\cos\theta_{n,k_0}-\cos\theta} \right|
+|d-1| \left|\frac{\cos n\theta}{n}\,\frac{\sin\theta_{n,k_0}}{\cos\theta_{n,k_0}-\cos\theta} \right|
  \\
  &\leq \frac{1+|d-1|}{n}\frac{1}{\cos\theta_{0}-\cos\theta_1}.
\end{align*}
It follows that $\left(L_n h\right)_{n\ge1}$ converges uniformly to $h$ in $[-1,a]$.

Now let
$b=\cos \theta_2\in ]x_0,1[$ and
$x=\cos \theta\in [b,1]$. For sufficiently large $n\ge1$   there
exists $k_0$ such that
\[
0\leq \theta\leq \theta_2<\theta_{n,k_0}
\leq \theta_0<\theta_{n,k_0+1}<2\pi
\]
and consequently
\[
0<\cos \theta_2-\cos\theta_0\leq  \cos \theta-\cos\theta_{n,k_0+1}\;.
\]
Then
\begin{align*}
|1-&L_nh(x)|=|1-L_nh(\cos \theta)|=\left| \sum_{k=1}^{n}\ell_{n,k}(\cos \theta)-
\sum_{k=1}^{k_0}\ell_{n,k}(\cos \theta)h(\cos \theta_{n,k})\right|\\
&=\left| \sum_{k=k_0+1}^{n}
  \frac{(-1)^{k}}{n}\,
\frac{\cos n\theta}{\cos\theta-\cos\theta_{n,k}}\,\sin\theta_{n,k}-
(d-1)\chi_{\left\{\theta_{n,k_0}\right\}}(\theta_0)\ell_{n,k_0}(\cos \theta)\right|\\
&\leq \left|
 \frac{\cos n \theta }{n}
\frac{\sin \theta_{n,k_0+1}} {\cos \theta- \cos \theta _{n,k_0+1}}
\right|+|d-1|\left|
 \frac{\cos n \theta }{n}
\frac{\sin \theta_{n,k_0}} {\cos \theta- \cos \theta _{0}}
\right|
\\
&\leq\frac{1+|d-1|}{n}\frac{1}{\cos\theta_2-\cos\theta_{0}}\;,
\end{align*}
since the function
$t\to \frac{\sin t}{\cos \theta-\cos t}$ is positive and monotone decreasing in
$]\theta,\pi]$ and
$\theta<\theta_{n,k-1}<\theta_{n,k}<\pi$ for every $k_0+1\leq k\leq n$.
So $\left(L_n h\right)_{n\ge1}$ converges uniformly to $h$ also in $[b,1]$.

Now, we study the behavior of $(L_nh(x_0))_{n\ge1}$.

For sufficiently large $n\ge1$ there exists  $k_0$ such that
$\theta_{n,k_0}\leq \theta_0<\theta_{n,k_0+1}$.
Let us denote $\displaystyle   \sigma_n=n\frac{\theta_0-\theta_{n,k_0}}{\pi}$.
{}From $\frac{2k_0-1}{2n}\pi \leq\theta_0<\frac{2k_0+1}{2n}\pi$ we have that
$0\leq \sigma_n <1$;
then
\[
n=\frac {\pi}{\theta_0}(\sigma_n+k_0-1/2)
\]
and moreover
\[
k_0\leq n\frac{\theta_0}{\pi}+\frac 12\leq k_0+1\;,
\]
that is $k_0=\left[ n\frac{\theta_0}{\pi}+\frac 12\right]$
and
\begin{equation} \label{eq:sigman}
\sigma_n=n\frac{\theta_0}{\pi}+\frac 12-\left[ n\frac{\theta_0}{\pi}+\frac 12\right]\;.
\end{equation}

If $x_0$ is a Chebyshev node, that is $\theta_0=\theta_{n,k_0}$ and $\sigma_n=0$, then
\begin{equation}\label{eq:dis}
L_nh(\cos \theta_0)=d\;.
\end{equation}
If $x_0$ is not a Chebyshev node we have
$\theta_0<\theta_{n,k_0}$,\ $0<\sigma_n<1$ and
\begin{equation}\label{eq:ln-nodis}
L_nh(\cos \theta_0)=\sum_{k=1}^{k_0}\ell_{n,k}(\cos \theta_0)\;.
\end{equation}

Let us consider the case where $x_0$ is not a Chebyshev node and observe that
\begin{align*}
\sum_{k=1}^{k_0}\ell_{n,k}(\cos \theta_0)
%L_n(h)(\cos \theta_0)
&=\sum_{k=1}^{k_0}
  \frac{(-1)^{1-k}}{n}\,\frac{\cos n\theta_0}{\cos\theta_0-\cos\theta_{n,k}}\,\sin\theta_{n,k}\\
&=\sum_{k=1}^{k_0}
  \frac{(-1)^{k_0-k+1}}{n}\,\frac{\sin (n\theta_0-k_0\pi+\pi/2)}{\cos\theta_0-\cos\theta_{n,k}}\,\sin\theta_{n,k}\\
&=\sum_{k=1}^{k_0}
\frac{(-1)^{k_0-k+1}}{n}\,\frac{\sin \left(n(\theta_0-\theta_{n,k_0})\right)}
{\cos\theta_0-\cos\theta_{n,k}}
\,\sin\theta_{n,k}\;.%\\
%&=\sum_{k=1}^{k_0}
%(-1)^{k_0-k}\,
%\frac{\sin \left(n(\theta_0-\theta_{n,k_0})\right)}{n(\theta_0-\theta_{n,k})}
%\frac{\frac12(\theta_0-\theta_{n,k})}
%{\sin\left(\frac12(\theta_0-\theta_{n,k}) \right)}
%\,
%\frac{\sin\left(\theta_0-(\theta_0-\theta_{n,k})\right)}{\sin\left(\theta_0-\frac12(\theta_0-\theta_{n,k})\right)}\;,
\end{align*}
Setting $m=k_0-k$ we have
\begin{align*}
\theta_0 -\theta_{n,k}&=\theta_0 -\theta_{n,k_0-m}=
\theta_0 -\frac{2(k_0-m)-1}{2n}\pi=\theta_0-\theta_{n,k_0}+\frac mn\pi
\\
&=\frac{\pi}{n}(\sigma_n +m)
\end{align*}
and consequently
\begin{align}\label{eq:sum1}
%&L_n(h)(\cos \theta_0)=
&\sum_{k=1}^{k_0}\ell_{n,k}(\cos \theta_0)=
\sum_{m=0}^{k_0-1}
\frac{(-1)^{m+1}}{n}\,
\frac{\sin \left( \pi \sigma_n \right)}
 {\cos\theta_0-\cos\theta_{n,k_0-m}}
\,\sin\theta_{n,k_0-m} \nonumber
\\
%&\quad:=f_n(\sigma_n)
&\quad=\frac{\sin(\pi \sigma_n)}{\pi}
  \sum_{m=0}^{k_0-1}\frac{(-1)^{m}}{\sigma_n+m}
\nonumber \\
&\qquad+\frac{\sin(\pi \sigma_n)}{n}\sum_{m=0}^{k_0-1}(-1)^{m+1}\,
 \left[\frac{\sin\theta_{n,k_0-m}} {\cos\theta_0-\cos\theta_{n,k_0-m}}
  +\frac n\pi\frac{1}{\sigma_n+m}\right]
\nonumber\\&\quad=\frac{\sin(\pi \sigma_n)}{\pi}
  \sum_{m=0}^{k_0-1}\frac{(-1)^{m}}{\sigma_n+m}
\nonumber \\
&\qquad+\frac{\sin(\pi \sigma_n)}{n}\sum_{m=0}^{k_0-1}(-1)^{m+1}\,
 \left[\frac{\sin\theta_{n,k_0-m}} {\cos\theta_0-\cos\theta_{n,k_0-m}}
  +\frac{1}{\theta_0-\theta_{n,k_0-m}}\right]
\nonumber\\&\quad=\frac{\sin(\pi \sigma_n)}{\pi}
  \sum_{m=0}^{k_0-1}\frac{(-1)^{m}}{\sigma_n+m}
\nonumber\\
&\qquad+\frac{\sin(\pi \sigma_n)}{n}
\sum_{m=0}^{k_0-1}(-1)^{m+1}\,
g_{\theta_0}(\theta_{n,k_0-m})
\end{align}
where
$\theta_{n,k_0-m}\in [\theta_{n,1},\theta_{n,k_0}]\subset ]0,\theta_0[$ and the function $g_{\theta_0}:\, ]0,\theta_0[\rightarrow \R$ is defined by setting
\[
g_{\theta_0}(x):=\frac{\sin x}{\cos\theta_0-\cos x}+\frac{1}{\theta_0-x}\;, \qquad x\in ]0,\theta_0[\;.
\]
The function $g_{\theta}$ is monotone decreasing and bounded since
\[
\lim_{x\rightarrow 0^+}g_{\theta_0}(x)=\frac{1}{\theta_0}<\infty\;,\qquad
\lim_{x\rightarrow \theta^{-}_{0}}g_{\theta_0}(x)=\frac 12 \cot(\theta_0)<\infty\;.
\]

For all $n\ge1$ and $\sigma \in [0,1[$, consider the function $f_n:[0,1[\rightarrow\R$ defined by setting
\begin{equation*}
f_{n}(\sigma):=
\left\{
\begin{array}{ll}
\displaystyle\frac{\sin(\pi \sigma)}{\pi}
\sum_{m=0}^{k_0-1}\frac{(-1)^{m}}{\sigma+m}\\
\qquad\displaystyle+\frac{\sin(\pi \sigma)}{n}\sum_{m=0}^{k_0-1}(-1)^{m+1}\,g_{\theta_0}(\theta_{n,k_0-m})\;,
&\text{if }\sigma \in ]0,1[\;,\\
d\;,&\text{if }\sigma=0\;;
\end{array}
\right.
\end{equation*}
taking into account \eqref{eq:dis}, \eqref{eq:ln-nodis} and
\eqref{eq:sum1} we have that
$L_n h(\cos \theta_0)=f_{n}(\sigma_n)$.

For all $\sigma \in ]0,1[$
\begin{align*}
|f_n(\sigma)-&g(\sigma)|
\\
&\leq\left|\frac{\sin(\pi\sigma)}{\pi}\sum^{\infty}_{m=k_0}\frac{(-1)^m}{\sigma+m} \right|+
\frac{\sin(\pi \sigma)}{n}
\left(|g_{\theta_0}(\theta_{n,1})|+|g_{\theta_0}(\theta_{n,k_0})| \right)
\\
&\leq \frac{\sin(\pi \sigma)}{\pi}\left|\frac{(-1)^{k_0}}{\sigma + k_0}\right|+
\frac{\sin(\pi \sigma)}{n}
\left(|g_{\theta_0}(\theta_{n,1})|+|g_{\theta_0}(\theta_{n,k_0})| \right)
\\
&\leq \frac{1}{\pi k_0}+
\frac{1}{n}
\left(|g_{\theta_0}(\theta_{n,1})|+|g_{\theta_0}(\theta_{n,k_0})| \right);
\end{align*}
where the righthand side is independent of $\sigma \in ]0,1[$ and
it converges to $0$ as $n\rightarrow \infty$ since
\[
\lim_{n\rightarrow \infty}g_{\theta_0}(\theta_{n,1})=
\lim_{x\rightarrow 0^+}g_{\theta_0}(x)=\frac{1}{\theta_0}<\infty
\]
and
\[
\lim_{n\rightarrow \infty} g_{\theta_0}(\theta_{n,k_0})=
\lim_{x\rightarrow {\theta_0}^-}g(x)=\frac 12 \cot(\theta_0)<\infty\;.
\]

Then we can conclude that the sequence $(f_n)_{n\ge1}$ converges uniformly on $[0,1[$ to the
function $\tilde g:[0,1[\rightarrow\R$ defined as follows
\begin{equation*}
\tilde{g}(x):=\left\{
\begin{array}{ll}
g(x)\;,\qquad&\text{if }x\in ]0,1[\;,
\\
d\;,&\text{if }x=0\;.
\end{array}
\right.
\end{equation*}

Now, we will construct $q$ subsequences
$\left(L_{k_m(n)}h(x_0)\right)_{n\ge1}$, $m=0,\dots,q-1$, of
$(L_{n}h(x_0))_{n\ge1}$ with density $\frac 1q$ such that
\[
\lim_{n\rightarrow\infty}L_{k_m(n)}h(x_0)=
\tilde{g}\left( \frac{m+q/2-[q/2]}{q}\right)
 \text{ for all } m=0,\dots, q-1\;.
\]

Fix $m=0,\dots,q-1$; since $GCD(p,q)=1$ we can set $k_m(n):=l+nq$, where $l\in \{-[q/2],\dots,[q/2]\}$ is such that
$lp\equiv m -[q/2] \mod q$, that is there exists $s\in \Z$ such that
$lp=sq+m-[q/2]$.

So, consider $\left(L_{k_m(n)}h(x_0)\right)_{n\ge1}$ and observe that
for all $m=0,\dots,q-1$, we have $\delta(\{k_m(n)\ |\ n\in \N\})=\frac 1q$. It follows, for all $n\ge1$
\begin{align*}
\sigma_{k_m(n)}&=(l+nq)\frac pq +\frac 12-\left[(l+nq)\frac pq +\frac 12\right]
\\
&=s+\frac{m+q/2-[q/2]}{q}+np-\left[ s+\frac{m+q/2-[q/2]}{q}+np \right]
\\
&=\frac{m+q/2-[q/2]}{q}
\end{align*}
since $s,np \in \Z$, while $0\leq \frac{m+q/2-[q/2]}{q}< 1$ because $0\leq q/2-[q/2]<1$.
%Then
%\begin{align*}
%&\lim_{n\rightarrow \infty}L_{k_m(n)}h(x_0)
%=\lim_{n\rightarrow \infty}f_n\left( \sigma_{k_m(n)}\right)\\
%&\quad =\lim_{n\rightarrow \infty}
%f_n \left( \frac{m+q/2-[q/2]}{q} \right)=g\left( \frac{m+q/2-[q/2]}{q} \right)\;.
%\end{align*}
Then
\begin{align*}
\lim_{n\rightarrow \infty}&L_{k_m(n)}h(x_0)
=\lim_{n\rightarrow \infty}f_n\left( \sigma_{k_m(n)}\right)\\
&\quad =\lim_{n\rightarrow \infty}
f_n \left( \frac{m+q/2-[q/2]}{q} \right)=\tilde{g}\left( \frac{m+q/2-[q/2]}{q} \right)\;.
\end{align*}

Therefore, by Proposition \ref{prop:char_index},
we have that
for all $m=1,\dots,q$
\[
i\left(L_nh(x_{0}),
\tilde{g}\left(\frac{m+q/2-[q/2]}{q}\right)\right)
\geq \frac 1q\;.
\]
Now, we have $q$ different statistical limits with index $\frac 1q$, so by Proposition \ref{pr:sum_gen}
it necessarily follows
\[
i\left(L_nh(x_{0});
\tilde{g}\left(\frac{m+q/2-[q/2]}{q}\right)\right)
=\frac 1q.
\]
In particular, if $q$ is even and $m=0$, for every $n\ge1$
we have $\theta_0=\theta_{k_m(n),k_0}$ and
\[
\lim_{n\rightarrow \infty}L_{k_m(n)}h(x_0)=\tilde{g}(0)=d.
\]
This completes the proof of part i).

Finally we consider the case where $\frac{\theta _0}{\pi}$ is irrational.
First, we observe that from \eqref{eq:sigman} and Example \ref{ex:cos_equi} (ii), we have
\[
\delta \left(\left\{n\in \N\ |\ \sigma_n\in J\right\}\right)=|J|
\]
for every Peano-Jordan measurable set $J\subset [0,1[$.

Let $A\subset \R$ be a bounded Peano-Jordan measurable set.
Since $(f_n)_{n\ge1}$ converges uniformly to $g$ in $]0,1[$, for every
$\varepsilon>0$ there exists $\nu_\varepsilon\ge 1$ such that $f_n(\sigma)\in A+B_\varepsilon$
whenever $n\geq \nu_\varepsilon$ and $\sigma \in g^{-1}(A)$.
So
\[
\left\{n\geq \nu_\varepsilon\ |\ \sigma_n\in g^{-1}(A)\right\}\subset
\left\{n\geq \nu_\varepsilon\ |\ L_{n}h(x_0)\in A+B_\varepsilon \right\}
\]
and we can conclude that
\[
|g^{-1}(A)|\leq
\delta_- \left(\left\{n\in \N \ |\ L_{n}h(x_0)\in A+B_\varepsilon\right\}\right)\;.
\]
Hence
\begin{equation}\label{eq:index_A}
|g^{-1}(A)|\leq i(L_n h(x_0);A)\;.
\end{equation}
In order to show the converse inequality, we argue by contradiction and assume that
$|g^{-1}(A)|<i(L_n h(x_0);A)$; then we can find $\delta>0$ such that
\[
|g^{-1}(A+B_\delta)|=i(L_n h(x_0);A)\;.
\]
The map $\delta\mapsto |A+B_\delta|$ is monotone increasing and continuous for $\delta\ge0$.

By Proposition \ref{pr:sum_gen}, since $\overline{A}\cap \overline{(A+B_{\delta/2})^c}=\emptyset$,
we have
\[
i(L_n h(x_0);A)+i(L_n h(x_0);(A+B_{\delta/2})^c)\leq 1.
\]
Using (\ref{eq:index_A}), we get $|g^{-1}((A+B_{\delta/2})^c)|\leq i(L_n h(x_0); (A+B_{\delta/2})^c)$; then
\[
|g^{-1}(A+B_\delta)|+|g^{-1}((A+B_{\delta/2})^c)|\leq 1
\]
and consequently, taking into account that $g^{-1}(\R)=]0,1[$,
\[
|g^{-1}(A+B_\delta)|\leq 1-|g^{-1}((A+B_{\delta/2})^c)|=|g^{-1}(A+B_{\delta/2})|
\]
which yields a contradiction, since the map $\delta\mapsto |A+B_\delta|$ is monotone increasing.
\qev
\end{proof}

At this point, using Theorem \ref{th:Lagrange}, we are able to study the behavior of Lagrange operators on larger classes of
functions, namely on the space BV([-1,1]) of functions of bounded variation having a finite number of points of discontinuity and on the space
$C_{\omega}+H$ where $C_\omega$ denotes the space of all functions
$f\in C([-1,1])$ satisfying the Dini-Lipschitz condition $\omega(f,\delta)=o(|\log\delta|^{-1})$,
and $H$ is the linear space generated by
\[
\{h_{x_0,d}\ |\ x_0\in]-1,1[,\ d\in\R\}\;.
\]

Observe that if $f\in C_\omega+H$ there exists at most a finite number of points $x_1,\dots,x_N$
of discontinuity with finite left and right limits $f(x_i-0)$ and $f(x_i+0)$,\ $i=1,\dots,N$.

Then we can state the following theorem.

\begin{theorem} \label{th:Lagrange_extended}
Let $f\in BV([-1,1])$, or alternatively $f\in C_\omega+H$, with a finite number $N$ of
points of discontinuity of the first kind at $x_1,\dots,x_N\in]-1,1[$.
For every $i=1,\dots,N$ consider $\theta_i\in]0,\pi[$ such that $x_i=\cos
\theta_i$,\ $d_i:=f(x_i)$ and define the function
\[
g_i(x):=f(x_i-0)+(f(x_i+0)-f(x_i-0))g(x)\;.
\]
%Let $x_0=\cos\theta_0\in[-1,1]$ and consider the functions $h$ defined by

Then, the sequence $\left( L_n f\right)_{n\ge1}$ converges uniformly
to $f$
on every compact subset of $]-1,1[\setminus\{x_1,\dots,x_N\}$.

Moreover for all $i=1,\dots,N$ the sequence $(L_n f(x_i))_{n\ge1}$ has the following
behavior
\begin{itemize}
\item[i)] if $\frac{\theta_i}{\pi}=\frac pq $ with $p,q\in \N$, $q\neq 0$ and
$GCD(p,q)=1$, then
  \[
    i\left(L_n f (x_i);
    g_i\left(\frac{2m+1}{2q}\right)\right)=\frac1q\;,
    \quad m=0,\dots, q-1;
  \]
if $q$ is odd and
  \[
    i\left(L_n f (x_i);d_i
  \right)=\frac1q\;,
  \]
  \[
    i\left(L_n f (x_i);
    g_i\left(\frac mq\right)
    \right)=\frac1q\;,
    \quad m=1,\dots, q-1.
  \]
if $q$ is even.

\item [ii)] If $\frac{\theta_i}{\pi}$ is irrational and if
$A\subset \R$ is a Peano-Jordan measurable set,
then
\[
i\left(L_nf(x_i);A\right)=
| g_i^{-1}(A)|
\;.
\]
\end{itemize}
\end{theorem}

\begin{proof}
We assume $x_1<\dots<x_N$. We can write $f=F+\sum _{k=1}^{N}c_kh_k$,
where $F\in BV([-1,1])\cap C([-1,1])$ or, alternatively, $F\in C_\omega$ and
$h_i:=h_{x_i,\tilde d_i}$ for every $i=1,\dots,N$.

Since $F$ is continuous we have
\begin{align*}
f(x_i+0)-\sum_{k=1}^{i-1}c_k-c_i=F(x_i+0)=F(x_i-0)=f(x_i-0)-\sum_{k=1}^{i-1}c_k\;,
\end{align*}
from which
\[
c_i=f(x_i+0)-f(x_i-0)
\]
and
\begin{equation}\label{eq:F(x_i)}
 F(x_i)=f(x_i-0)-\sum_{k=1}^{i-1}c_k\;.
\end{equation}
Moreover
\begin{align*}
d_i=f(x_i)&=F(x_i)+\sum^{i-1}_{k=1}c_k h_k(x_i)+c_i \tilde d_i
\\
&=F(x_i)+\sum^{i-1}_{k=1}c_k +\left(f(x_i+0)-f(x_i-0)\right)\tilde d_i
\\
&=f(x_i-0)+\left(f(x_i+0)-f(x_i-0)\right)\tilde d_i\;.
\end{align*}
and hence
\[
\tilde d_i=\frac{d_i-f(x_i-0)}{f(x_i+0)-f(x_i-0)}\;.
\]

The first part of our statement is a trivial consequence of the linearity of
Lagrange interpolation operators. Indeed $L_n F\to F$ uniformly on compact subsets of $]-1,1[$
by \cite[Theorem 3.1, p. 24]{V} (see also \cite{MM})
if $F\in BV([-1,1])\cap C([-1,1])$ and by
\cite[Theorem 14.4, p. 335]{S} in the case $F\in C_{\omega}$).

Moreover for every $k=1,\dots,N$, by Theorem \ref{th:Lagrange}
$L_n h_k\to h_k$ converges uniformly to $h_k$ on compact subsets
of $[-1,1]\setminus \{x_k\}$. Then $L_n f=L_n F+\sum^{N}_{k=1}c_k L_n h_k$ converges
uniformly to $f$ on compact subsets of $]-1,1[\setminus\{x_1,\dots,x_N\}$.

Now we establish property i). We fix a point $x_i$ of discontinuity
and following the same line of the proof of Theorem \ref{th:Lagrange}
we construct the subsequences $(k_m(n))_{n\ge1}$,\ $m=1,\dots,q$.
Since
\[
L_{k_m(n)}f(x_i)=L_{k_m(n)} F(x_i)+
\sum_{
\genfrac{}{}{0pt}{}{k=1}{k\neq i}
}^{N}c_k L_{k_m(n)} h_k(x_i)+c_i
L_{k_m(n)} h_i(x_i)
\]
and taking into account \eqref{eq:F(x_i)} and that $F\in BV([-1,1])\cap C([-1,1])$
(or alternatively $F\in C_{\omega}$), from Theorem \ref{th:Lagrange}
the right-hand side converges to
\begin{align*}
&F(x_i)+\sum^{i-1}_{k=1}c_k h_k(x_i)+c_i g\left(\frac{2m+1}{2q}\right) \\
&\qquad =f(x_i-0)+\left(f(x_i+0)-f(x_i-0)\right)g\left(\frac{2m+1}{2q}\right)\\
&\qquad =g_i\left(\frac{2m+1}{2q}\right)
\end{align*}
for $m=0,\dots,q-1$, if $q$ is odd.

Analogously, if $q$ is even, $\left(L_n f(x_i)\right)_{n\ge1}$
converges to
\[
f(x_i-0)+\left(f(x_i+0)-f(x_i-0)\right)\tilde d_i=d_i
\]
with index $\frac 1q$ and to
\[
f(x_i-0)+\left(f(x_i+0)-f(x_i-0)\right)g\left(\frac mq\right)=g_i\left(\frac mq\right)
\]
with index $\frac 1q$ for $m=1,\dots,q-1$.

Finally, we prove property ii).
For every $i=1,\dots,N$ we have
\[
L_{n}f(x_i)=L_{n} F(x_i)+
\sum_{
\genfrac{}{}{0pt}{}{k=1}{k\neq i}
}^{N}
c_k L_{n} h_k(x_i)+c_i
L_{n} h_i(x_i)\;.
\]
For the sake of simplicity let us denote
\[
y_n:=L_n f(x_i)\;,\quad
z_n:=L_n F(x_i)+
  \sum_{\genfrac{}{}{0pt}{}{k=1}{k\neq i}}^{N}c_k L_n h_k(x_i)\;,\quad
x_n:=c_iL_n h_i(x_i)\;;
\]
thus $y_n=z_n+x_n$ and, from \eqref{eq:F(x_i)},
\[z:=F(x_i)+\sum_{k=1}^{i-1}c_k
h_k(x_i)=f(x_i-0)\;.
\]

Now, we can apply \cite[Theorem 3.1, p. 24]{V} if $F\in BV([-1,1])\cap C([-1,1])$
and \cite[Theorem 14.4, p. 335]{S} if $F\in C_{\omega}$ and in any case, from
Theorem \ref{th:Lagrange}, we get $z_n\rightarrow z$ and
$i(c_i^{-1}x_n;A)=|g^{-1}(A)|$ for every bounded Peano-Jordan measurable
set $A\subset \R$. Hence $i(x_n;A)=|g^{-1}(c_i^{-1}A)|$, that is
\[|g^{-1}(c_i^{-1}A)|
=\inf_{\varepsilon>0}
\delta_-(\{n\in \N\ |\ x_n\in A+B_\varepsilon\})\;.\]

Fix $\varepsilon>0$; if $x_n\in A+B_\varepsilon$, from the equality $x_n=y_n-z_n$ we get
\[
y_n\in A+B_\varepsilon+z_n=A+B_\varepsilon+z+z_n-z\;.
\]
Now, let $\nu \in \N$ such that $|z_n-z| < \varepsilon$ for all $n\geq \nu$,
Then for every
$n\geq \nu$ we have $z_n-z\in B_\varepsilon$ and consequently $y_n\in A+B_{2\varepsilon}+z$.
Then
\begin{align*}
\{n\geq \nu\ |\ x_n\in A+B_\varepsilon\}\subset \{n\geq \nu\ |\ y_n\in
A+B_{2\varepsilon}+z\},
\end{align*}
that is
\begin{align}\label{eq:111}
\delta_- (\{n\in \N\ |\ x_n\in A+B_\varepsilon\})
\leq
\delta_-( \{n\in \N\ |\ y_n\in A+B_{2\varepsilon}+z\})\;.
\end{align}

On the other hand, if $y_n\in A+B_{2\varepsilon}+z$,
then $x_n=y_n-z_n\in A+B_{2\varepsilon}+z-z_n$.
In this case for every $n\geq \nu$, we have
$z-z_n\in B_\varepsilon$ and therefore
$x_n\in A+B_{3\varepsilon}$; hence
\begin{align}\label{eq:222}
\delta_-
(\{n\in \N\ |\ x_n\in A+B_{3\varepsilon}\})\geq
\delta_- \{n\in \N\ |\ y_n\in A+B_{2\varepsilon}+z\}\;.
\end{align}

Taking the infimum over $\varepsilon>0$
in \eqref{eq:111} and \eqref{eq:222}
we can conclude that
$i(x_n,A)\leq i(y_n,A+z)\leq i(x_n,A)$ which yields
\[
i(y_n,A+z)=i(x_n,A)=|g^{-1}(c_i^{-1}A)|\;.
\]
We conclude that
$i(y_n,A)=|g^{-1}(c_i^{-1}(A-z))|
=\left|g^{-1}\left(\frac{A-f(x_i-0)}{f(x_i+0)-f(x_i-0)}\right)\right|=|g_i^{-1}(A)|$
for every bounded Peano-Jordan measurable set $A\subset \R$.
\qev
\end{proof}

\section{Shepard operators on discontinuous functions} \label{sc:Shepard}
Let $s\ge1$; the $n$-th Shepard  operator $S_{n,s}$ %:C([-1,1])\to C([-1,1])
is defined by
\[
S_{n,s}f(x)=\frac{\sum _{k=0}^n f\left(\frac kn\right)
  \left|x-\frac kn\right|^{-s}        }
{\sum _{k=0}^n  \left|x-\frac kn\right|^{-s}}
\]
for every %$f\in C([-1,1])$
$f:\, [0,1]\rightarrow \R$ and $x\in[0,1]$.

For the general properties of Shepard operators we refer to \cite{DVM}. In particular
we pont out that the sequence $\left( S_{n,s}f\right)_{n\ge1}$ converges uniformly to $f$ for
every $f\in C([-1,1])$ (see \cite[Theorem 2.1]{DVM}).

Our aim is to study the behavior of the sequence of Shepard operators for
bounded functions which have a finite number of points of discontinuity
of the first kind and are continuous elsewhere.

Also in this case we begin by considering the function
$h_{x_0,d}:[0,1]\rightarrow\R$ defined by
\begin{equation} \label{eq:def_hS}
h_{x_0,d}(x):=\left\{
\begin{array}{ll}
1\;,\quad& x<x_0\;,\\
d\;,\quad& x=x_0\;,\\
0\;,\quad& x>x_0\;,
\end{array} \right. \qquad x\in[0,1]\;,
\end{equation}
where $x_0\in[0,1]$ and $d\in\R$ are fixed.

In order to state the convergence properties of the sequence $(S_{n,s}h_{x_0,d})_{n\ge1}$,
for every $s>1$ we consider the function $g_s:[0,1]\mapsto \R$ defined by setting, for every $x\in [0,1]$,
\[
g_s(x)=\frac{\zeta(s,x)}{\zeta(s,x)+\zeta(s,1-x)}\;,
\]
where $\zeta$ is the Hurwitz zeta function defined by \eqref{eq:def_H}.

We have the following result.

\begin{theorem} \label{th:shepard}
Let $x_0\in[0,1]$ and $h:=h_{x_0,d}$ be defined by \eqref{eq:def_hS}.
Then for every $s\ge1$ the sequence $\left( S_{n,s}h\right)_{n\ge1}$ converges uniformly to $h$
on every compact subset of $[0,1]\setminus\{x_0\}$.

As regards the behavior of the sequence $(S_{n,s}h(x_0))_{n\ge1}$ we have
\begin{itemize}
\item[i)]
If $x_0=\frac pq $ with $p,q\in \N$, $q\neq 0$ and $GCD(p,q)=1$, then
\[
i\left(S_{n,s}h(x_0);d\right)=\frac 1q
\]
and further
\begin{align*}
s>1 \ \Longrightarrow \ &   i\left(S_{n,s}h(x_0);g_s\left( \frac{m}{q}\right)\right)=\frac 1q\;,\qquad m=1,\dots, q-1\;,\\
s=1 \ \Longrightarrow \ &   i\left(S_{n,s}h(x_0);\frac12\right)=1-\frac 1q\;.
\end{align*}

\item [ii)] if $x_0$ is irrational and if
$A\subset \R$ is a Peano-Jordan measurable set, then
\begin{align*}
s>1 \ \Longrightarrow \ &   i\left(S_{n,s}h(x_0);A\right)=|g_s^{-1}(A)|\;,\\
s=1 \ \Longrightarrow \ &   i\left(S_{n,s}h(x_0);\frac12\right)=1\;.
\end{align*}
Moreover, in the case $s=1$, there exist subsequences of $(S_{n,s}h(x_0))_{n\ge1}$ converging to $0$ and $1$
(consequently the set of indices of these subsequences must have density zero).
\end{itemize}
\end{theorem}

\begin{proof}
We set $k_0=[nx_0]$, so that, for sufficiently large $n\ge1$,
$\frac{k_0}{n}\leq x_0<\frac{k_0+1}{n}$.

Let $a\in [0,x_0[$ and $x \in [0,a]$,
we have
\[
S_{n,s}h(x)=\frac{
  \sum _{k=0}^{k_0} \left|x-\frac kn\right|^{-s}
+\chi_{\{x_0 n\}}\left( k_0 \right) \left( d-1 \right)
\left|x-x_0\right|^{-s}}
{\sum _{k=0}^n  \left|x-\frac kn\right|^{-s}}\;,
\]
then
\begin{align}\label{eq:frac}
&S_{n,s}h(x)-1\nonumber \\
&\quad=\frac{
  \sum _{k=0}^{k_0} \left|x-\frac kn\right|^{-s}
+\chi_{\{x_0 n\}}\left( k_0 \right) \left( d-1 \right)
\left|x-x_0\right|^{-s}
-\sum _{k=0}^n  \left|x-\frac kn\right|^{-s}}
{\sum _{k=0}^n  \left|x-\frac kn\right|^{-s}}\nonumber \\
&\quad=\frac{
  \sum _{k=k_0+1}^{n} \left|x-\frac kn\right|^{-s}
+\chi_{\{x_0 n\}}\left( k_0 \right) \left( d-1 \right)
\left|x-x_0\right|^{-s}}
{\sum _{k=0}^n  \left|x-\frac kn\right|^{-s}}\;.
\end{align}
If $k> k_0$, since $x\leq a<x_0<\frac kn$, we have that
$\frac kn-x>x_0-a>0$, moreover $x_0-x>x_0-a$,
then
\begin{align}\label{eq:num}
\sum _{k=k_0+1}^{n} &\left|x-\frac kn\right|^{-s}
+\chi_{\{x_0 n\}}\left( k_0 \right) |d-1|
\left|x-x_0\right|^{-s}\nonumber \\
&\leq (n-k_0)|x_0-a|^{-s}+\chi_{\{x_0 n\}}\left( k_0 \right) | d-1|
\left|x_0-a\right|^{-s}\nonumber \\
&\leq (n-[nx_0]+|d-1|)\left|x_0-a\right|^{-s}<+\infty\;.
\end{align}
On the other hand we have
\begin{align}\label{eq:den}
\sum_{k=0}^{n}\left|x-\frac kn \right|^{-s}&\geq
 \sum_{[nx]< k \leq n}\left|x-\frac kn \right|^{-s}
 =n^s \sum_{[nx]< k \leq n}(k-nx)^{-s}\nonumber \\
&\geq n^s \sum_{[nx]< k \leq n}(k-[nx])^{-s}
=n^s \sum_{m=1}^{n-[nx]}m^{-s}
\geq n^s \sum_{m=1}^{n-[na]}m^{-s}\;.
\end{align}
{}From \eqref{eq:num} and \eqref{eq:den}, we can rewrite \eqref{eq:frac} as follows
\begin{align*}
 &|S_{n,s}h(x)-1|\leq
\frac{(n-[nx_0]+|d-1|)\left|x_0-a\right|^{-s}}
  {n^s \sum_{m=1}^{n-[na]}m^{-s}}
\end{align*}
where the righthand side converges to $0$ as $n \to \infty$.
Indeed if $s>1$,
\[\lim_{n\to \infty} \frac{(n-[nx_0]+d)}{n^s}=0\;,\] while if
$s=1$
\[\lim_{n\to \infty}\sum_{m=1}^{n-[na]}m^{-1}=+\infty\;.\]

Now let $x_0<a\leq x \leq 1$; using the same arguments we get
\begin{align*}
 &|S_{n,s}h(x)|\leq
\frac{([nx_0]+1+|d-1|)\left|x_0-a\right|^{-s}}
  {n^s \sum_{m=1}^{[na]+1}m^{-s}}\;.
\end{align*}
Then the sequence of functions $\left( S_{n,s}h\right)_{n\ge1}$ converges uniformly to $h$ on every compact subset of $[0,1]\setminus\{x_0\}$.

Now we focus our attention on the behavior of the sequence $(S_{n,s}h(x_0))_{n\ge1}$.
Let us denote
$\displaystyle   \sigma_n=nx_0-k_0$, that is
$\displaystyle   \sigma_n=nx_0-[nx_0]$ and observe that $0\leq \sigma_n<1$.
If $x_0$ coincides with a node then
\[S_{n,s}h(x_0)=d\] otherwise
\begin{align*}
S_{n,s}h(x_0)&=\frac{\sum _{k=0}^{k_0} \left|x_0-\frac kn\right|^{-s} }
{\sum _{k=0}^n  \left|x_0-\frac kn\right|^{-s}}\\
&=\frac{\sum _{k=0}^{k_0} \left|nx_0-k\right|^{-s} }
{\sum _{k=0}^n  \left|nx_0- k\right|^{-s}}
=
\frac{\sum _{k=0}^{k_0} \left|nx_0-k_0+k_0-k\right|^{-s} }
{\sum _{k=0}^n  \left|nx_0-k_0+k_0-k\right|^{-s}}\\
&=\frac{\sum _{m=0}^{k_0} \left(\sigma_n+m\right)^{-s} }
{
\sum _{m=0}^{k_0} \left(\sigma_n+m\right)^{-s}+
\sum _{m=0}^{n-k_0-1} \left(1-\sigma_n+m\right)^{-s}
}\;.
\end{align*}
Now, consider the function $f_{n,s}:[0,1[\rightarrow\R$ defined by setting
\begin{equation*}
f_{n,s}(\sigma):=
\left\{
\begin{array}{ll}
\displaystyle\frac{\sum _{m=0}^{k_0} \left(\sigma+m\right)^{-s} }
  {\sum _{m=0}^{k_0} \left(\sigma+m\right)^{-s}+
    \sum _{m=0}^{n-k_0-1} \left(1-\sigma+m\right)^{-s}}\;,
  &\text{ if }\sigma \in ]0,1[\;, \\
\\
  d\;, &\text{ if }\sigma=0\;.
\end{array}
\right.
\end{equation*}
We have $S_{n,s}h(x_0)=f_{n,s}(\sigma_n)$.

If $s>1$, the sequence $\left( f_{n,s} \right)_{n\ge1}$ converges uniformly
to the function $g^*_s$ given by
\begin{equation*}
g^*_s(\sigma )=
\left\{
\begin{array}{ll}
d\;, &\text{ if }\sigma=0\;,\\
g_s(\sigma)\;, &\text{ if } \sigma \in ]0,1[ \;.
\end{array}
\right.
\end{equation*}
If $s=1$ the sequence $\left( f_{n,s} \right)_{n\ge1}$ converges pointwise to
\begin{equation*}
g^*_1(\sigma )=
\left\{
\begin{array}{ll}
d\;, &\text{ if }\sigma=0\;,
\\
\frac12\;, &\text{ if } \sigma \in ]0,1[
%\\
%0 &\text{ if }\sigma=1
\end{array}
\right.
\end{equation*}
and the convergence is uniform on every compact subset of $]0,1[$.

Arguing similarly to the proof of Theorem \ref{th:Lagrange}
we obtain property $i)$ for $s\geq 1$ and $ii)$ for $s>1$.

As regards the case $s=1$, we consider an interval $[a,b]\subset ]0,1[$. Since the sequence
$(f_{n,1})_{n\ge1}$ converges uniformly in $[a,b]$, for every $\varepsilon > 0$
there exists $\nu\in \N$ such that $|f_{n,1}(x)-\frac 12|\leq \varepsilon$
whenever $n\geq \nu$ and $x\in [a,b]$.
Then
\[
\{n\in\N\ |\ n\geq \nu, \sigma_n \in [a,b]\}
\subset
\left\{n\in \N\ |\ S_{n,1}h(x_0)\in \left]\frac 12-\varepsilon,\frac 12+\varepsilon
\right[\right\}\;,
\]
and since the sequence $\sigma _n$ is equidistributed  on $]0,1[$
\[\delta_-\left\{n\in\N\ |\ S_{n,1}h(x_0)\in
\left]\frac 12-\varepsilon,\frac 12+\varepsilon\right[\right\}\geq b-a\]
that is $i(S_{n,1}h(x_0);\frac 12)\geq b-a$ for every $0<a<b<1$.

It follows $i(S_{n,1}h(x_0);\frac 12)=1$.

Moreover from \cite[Theorem 2]{BDVM} we have that
\[
\limsup_{n\to \infty} S_{n,1}h(x_0)=\max \left\{\lim_{x\to x_0^-}h(x),\lim_{x\to x_0^+}h(x)\right\}=1
\]
and
\[
\liminf_{n\to \infty} S_{n,1}h(x_0)=\min \left\{\lim_{x\to x_0^-}h(x),\lim_{x\to x_0^+}h(x)\right\}=0\;,
\]
then we can construct subsequences of $(S_{n,1}h(x_0))_{n\ge1}$ converging to $0$ and $1$, but,
thanks to Remark \ref{rm:sum=1}
and Proposition \ref{prop:char_index}, they must have the
set of indices with density zero.
\qev
\end{proof}

Finally, we extend Theorem \ref{th:shepard} to a larger class of
functions.

\begin{theorem} \label{th:Shepard_extended}
Let $f$ be a bounded function with $N$ points of discontinuity of the first kind at $x_1,\dots,x_N\in]0,1[$
and continuous elsewhere.
For every $i=1,\dots,N$ consider $d_i:=f(x_i)$ and define the function
\[
g_{s,i}(x):=f(x_i+0)+(f(x_i-0)-f(x_i+0))g_s(x)\;.
\]

Then, for every $s\ge1$ the sequence $\left( S_{n,s} f\right)_{n\ge1}$ converges uniformly
to $f$
on every compact subset of $[0,1]\setminus\{x_1,\dots,x_N\}$.

Moreover for all $i=1,\dots,N$ the sequence $(S_{n,s} f(x_i))_{n\ge1}$ has the following
behavior
\begin{itemize}
\item[i)]
If $x_i=\frac pq $ with $p,q\in \N$, $q\neq 0$ and $GCD(p,q)=1$, then
\[
i\left(S_{n,s}h(x_i);d_i\right)=\frac 1q
\]
and further
\begin{align*}
s>1 \ \Longrightarrow \ &   i\left(S_{n,s}f(x_i);g_{s,i}\left( \frac{m}{q}\right)\right)=\frac 1q\;,\qquad m=1,\dots, q-1\;,\\
s=1 \ \Longrightarrow \ &   i\left(S_{n,s}f(x_i);\frac{f(x_i+0)+f(x_i-0)}2\right)=1-\frac 1q\;.
\end{align*}

\item [ii)] if $x_i$ is irrational and if
$A\subset \R$ is a Peano-Jordan measurable set, then
\begin{align*}
s>1 \ \Longrightarrow \ &   i\left(S_{n,s}f(x_i);A\right)=|g_{s,i}^{-1}(A)|\;,\\
s=1 \ \Longrightarrow \ &   i\left(S_{n,s}f(x_i);\frac{f(x_i+0)+f(x_i-0)}2\right)=1\;.
\end{align*}
Moreover, in the case $s=1$, there exist subsequences of $(S_{n,s}f(x_i))_{n\ge1}$ converging to $f(x_i-0)$
and $f(x_i+0)$ whose set of indices has density zero.
\end{itemize}
\end{theorem}

\begin{proof}
We assume $x_1<\dots<x_N$. For every $k=1,\dots,N$, we set $c_k:=f(x_k-0)-f(x_k+0)$ and
$\tilde{d_k}:=\frac{d_k-f(x_k+0)}{f(x_k-0)-f(x_k+0)}$; consequently we can write $f=F+\sum _{k=1}^{N}c_kh_k$,
where $F\in C([0,1])$ and $h_k:=h_{x_k,\tilde d_k}$ for every $k=1,\dots,N$.
So $f\in C([0,1])+H$ and since Shepard operators converge uniformly in $C([0,1])$ (see e.g. \cite[Theorem 2.1]{DVM}),
we can argue as in Theorem \ref{th:Lagrange_extended} using Theorem \ref{th:shepard} in place of
Theorem \ref{th:Lagrange}.
\qev
\end{proof}

\end{document}